\documentclass[10pt]{amsart}
\usepackage{fullpage}
\usepackage{amssymb,amsmath,amsfonts,amsthm,mathrsfs}
\usepackage{color}
\definecolor{backgroundcolor}{rgb}{1,1,0.8}

\numberwithin{equation}{section}

\renewcommand{\AA}{\mathbb A}

\newcommand{\CC}{\mathbb C}

\newcommand{\FF}{\mathbb F}
\newcommand{\GG}{\mathbb G}

\newcommand{\QQ}{\mathbb Q}

\newcommand{\ZZ}{\mathbb Z}

\newcommand{\OO}{\mathcal O}

\newcommand{\calH}{\mathcal H}
\newcommand{\calP}{\mathcal P}

\newcommand{\calS}{\mathcal S}

\newcommand{\p}{\mathfrak p}
\newcommand{\frakq}{\mathfrak q}

\def\ab{{\operatorname{ab}}}

\def\Spec{\operatorname{Spec}} 
\def\Gal{\operatorname{Gal}}
 
\def \GL {\operatorname{GL}}

\def\Aut{\operatorname{Aut}} 

\def\Frob{\operatorname{Frob}}
\newcommand{\tors}{{\operatorname{tors}}}

\def\Tr{\operatorname{Tr}}

\def\cl{\operatorname{cl}}

\def\bbar#1{\setbox0=\hbox{$#1$}\dimen0=.2\ht0 \kern\dimen0 
\overline{\kern-\dimen0 #1}}
\newcommand{\Qbar}{{\overline{\mathbb Q}}} 
\newcommand{\Kbar}{\bbar{K}}

\newcommand{\FFbar}{\overline{\FF}}

\newtheorem{thm}{Theorem}[section]
\newtheorem{lemma}[thm]{Lemma}

\newtheorem{prop}[thm]{Proposition}

\newtheorem*{conj}{Conjecture}

\theoremstyle{definition}

\theoremstyle{remark}
\newtheorem{remark}[thm]{Remark}

\newenvironment{alphenum}{\hfill \begin{enumerate} }{\end{enumerate}}

\definecolor{webcolor}{rgb}{0.8,0,0.2}
\definecolor{webbrown}{rgb}{.6,0,0}
\usepackage[
        colorlinks,
        linkcolor=webbrown,  filecolor=webcolor,  citecolor=webbrown, 
        backref,
        pdfauthor={David Zywina}, 
       pdftitle={Abelian varieties over large algebraic fields with infinite torsion},
]{hyperref}
\usepackage[alphabetic,backrefs,lite]{amsrefs} 

\begin{document}
\title[Abelian varieties over large algebraic fields with infinite torsion]{Abelian varieties over large algebraic fields with infinite torsion}

\author{David Zywina}
\address{Department of Mathematics, University of Pennsylvania, Philadelphia, PA 19104-6395, USA}
\email{zywina@math.upenn.edu}
\urladdr{http://www.math.upenn.edu/\~{}zywina}

\begin{abstract}
Let $A$ be an abelian variety of positive dimension defined over a number field $K$ and let $\Kbar$ be a fixed algebraic closure of $K$.    For each element $\sigma$ of the absolute Galois group $\Gal(\Kbar/K)$, let $\Kbar(\sigma)$ be the fixed field of $\sigma$ in $\Kbar$.   We shall prove that the torsion subgroup of $A(\Kbar(\sigma))$ is infinite for all $\sigma\in \Gal(\Kbar/K)$ outside of some set of Haar measure zero.     This proves the number field case of a conjecture of Geyer and Jarden from 1978.
\end{abstract}


\maketitle

\section{Introduction}

Let $A$ be an abelian variety of positive dimension defined over a number field $K$.   The Mordell-Weil group $A(K)$ is finitely generated while the group $A(\Kbar)$, with $\Kbar$ a fixed algebraic closure of $K$, has infinite rank and infinitely many torsion points.   It is interesting to bridge this gap and study the structure of the groups $A(L)$ for various large algebraic extensions $L$ of $K$.   For example, the group $A(K^\ab)$ has finite torsion if and only if $A$ has no abelian subvarieties with complex multiplication over $K$, where $K^\ab$ is the maximal abelian extension of $K$ (\cite{MR885780}).

Let $\Gal_K$ be the absolute Galois group $\Gal(\Kbar/K)$.   Fix an integer $e\geq 1$.   The group $\Gal_K^e$ is profinite and is thus equipped with a unique Haar measure $\mu_K$ for which $\mu_K(\Gal_K^e)=1$.     For each $\sigma=(\sigma_1,\ldots,\sigma_e)\in \Gal_K^e$, let $\Kbar(\sigma)$ be the fixed field of $\sigma_1,\ldots, \sigma_e$ in $\Kbar$.   In this paper, we will consider the fields $\Kbar(\sigma)$ for almost all $\sigma$ in $\Gal_K^e$.   By ``almost all'', we mean for all $\sigma \in \Gal_K^e$ outside of some set with Haar measure $0$.   For almost all $\sigma\in \Gal_K^e$, the field $\Kbar(\sigma)$ is is {pseudo-algebraically closed} \cite{MR868860}*{Theorem 16.18} (i.e., every absolutely irreducible variety defined over $\Kbar({\sigma})$ has a $\Kbar({\sigma})$-rational point) and the absolute Galois group $\Gal_{\Kbar(\sigma)}$ is isomorphic to the free profinite group on $e$ generators \cite{MR868860}*{Theorem 16.13}.

Frey and Jarden showed that the group $A(\Kbar(\sigma))$ has infinite rank for almost all $\sigma\in \Gal_K^e$ \cite{MR0337997}*{Theorem 9.1}.  We will thus focus on the torsion points of $A(\Kbar(\sigma))$.    Jacobson and Jarden showed that if $e\geq 2$, then $A(\Kbar(\sigma))_{\tors}$ is finite for almost all $\sigma\in \Gal_K$ \cite{MR1831454}.   Our main theorem deals with the remaining case $e=1$.

\begin{thm} \label{T:main}
Let $A$ be an abelian variety of positive dimension defined over a number field $K$.   For all $\sigma\in \Gal_K$ outside a set of Haar measure zero, the group of torsion points in $A(\Kbar(\sigma))$ is infinite.
\end{thm}

Since there are only countable many abelian varieties defined over $K$, the set of measure zero in Theorem~\ref{T:main} can actually be chosen independent of $A$.

For each positive integer $m$ and field extension $L/K$, let $A(L)[m]$ be the $m$-torsion subgroup of $A(L)$.   Jacobson and Jarden have also shown that for almost all $\sigma\in \Gal_K$, the group $A(\Kbar(\sigma))[\ell^\infty]:=\bigcup_{n\geq1} A(\Kbar(\sigma))[\ell^n]$ is finite for all rational primes $\ell$ \cite{MR1831454}.   So to prove Theorem~\ref{T:main}, we will need to demonstrate that for almost all $\sigma\in \Gal_K$, the group $A(\Kbar(\sigma))[\ell]$ is non-zero for infinitely many primes $\ell$.

A weaker version of Theorem~\ref{T:main} was proved by Geyer and Jarden in \cite{MR2111902} where they first needed to replace $K$ by some finite extension (which may depend on $A$).  Our theorem, with the earlier results cited above, completes the proof of the following conjecture of Geyer and Jarden in the case where $K$ is a number field, see \cite{MR516151}.

\begin{conj}[Geyer-Jarden] Let $A$ be an abelian variety of positive dimension defined over a finitely generated field $K$ and let $e$ be a positive integer.  Then for almost all $\sigma\in \Gal_K^e$, we have:
\begin{alphenum}
\item \label{C:GJa} If $e=1$, then $A(\Kbar(\sigma))_{\tors}$ is infinite.
\item \label{C:GJb} If $e\geq 2$, then $A(\Kbar(\sigma))_{\tors}$ is finite.
\item \label{C:GJc} The group $A(\Kbar(\sigma))[\ell^\infty]$ is finite for each prime $\ell$.
\end{alphenum}
\end{conj}

Geyer and Jarden made this conjecture after proving if for the special case of an elliptic curve.  Following the approach of our main theorem, one should be able to prove this conjecture in the case where $K$ is a general finitely generated field of characteristic 0 (parts (\ref{C:GJb}) and (\ref{C:GJc}) are already known).  The only thing stopping us from doing so is the lack of a convenient reference for the image of Galois representations over such fields.

\subsection{Galois representations}
Throughout this section, we will let $A$ be an abelian variety of dimension $g\geq 1$ defined over a number field $K$.     For each prime $\ell$, the group $A(\Kbar)[\ell]$ is isomorphic to $\FF_\ell^{2g}$ and has an action of $\Gal_K$ that respects the group structure.   This Galois action thus defines a Galois representation
\[
\rho_{A,\ell}\colon \Gal_K \to \Aut(A(\Kbar)[\ell])\cong \GL_{2g}(\FF_\ell).
\]
Observe that for $\sigma\in \Gal_K$, we have $A(\Kbar(\sigma))[\ell]\neq 0$ if and only if the matrix $\rho_{A,\ell}(\sigma) \in \GL_{2g}(\FF_\ell)$ has $1$ as an eigenvalue.   Theorem~\ref{T:main} will be a straightforward application of the following proposition.

\begin{prop}  \label{P:key}  Let $A$ be an abelian variety of positive dimension defined over a number field $K$.  Then there is a finite Galois extension $L/K$, a set $\calS$ of rational primes with positive density, and a positive constant $c$ such that that the following hold:
\begin{alphenum}
\item \label{a:lower} For each prime $\ell \in \calS$ and $\beta\in \Gal_K$, we have
\[
\frac{|\{h \in \rho_{A,\ell}(\beta\Gal_L) :  \det(I-h)= 0\}|}{|\rho_{A,\ell}(\beta\Gal_K)|} \geq  \frac{c}{\ell}.
\]
\item \label{b:independence} The homomorphism $\prod_{\ell \in \calS}\rho_{A,\ell}\colon \Gal_L \to {\displaystyle\prod_{\ell\in \calS} \rho_{A,\ell}(\Gal_L)}$ is surjective.

\end{alphenum}
\end{prop}

Let us now explain how Theorem~\ref{T:main} follows from Proposition~\ref{P:key}.  We first define the measure $\mu=[L:K] \mu_K$ on $\Gal_K$, i.e., the Haar measure on $\Gal_K$ such that $\mu(\Gal_K)=[L:K]$.  Now fix any element $\beta \in \Gal_K$.  Since $\mu(\beta\Gal_L)=1$, we may view $\beta\Gal_L$ with measure $\mu$ as a probability space.  For each prime $\ell\in\calS$, define the set $U_{\ell} := \{\sigma \in \beta\Gal_L : A(\Kbar(\sigma))[\ell] \neq 0 \}$.    Since $A(\Kbar(\sigma))[\ell] \neq 0$ is equivalent to $\det(I-\rho_{A,\ell}(\sigma))=0$,  the set $U_\ell$ is thus $\mu$-measurable with
\[
\mu(U_\ell) = \frac{|\{h \in \rho_{A,\ell}(\beta\Gal_L) :  \det(I-h)= 0\}|}{|\rho_{A,\ell}(\beta\Gal_K)|}.
\]
Using Proposition~\ref{P:key}(\ref{b:independence}), we find that the map $\prod_{\ell\in\calS}\rho_{A,\ell}\colon \beta\Gal_L\to  \prod_{\ell\in \calS} \rho_{A,\ell}(\beta\Gal_L)$ is surjective, and thus the $U_\ell$ are $\mu$-independent subsets of $\beta\Gal_L$ (i.e., $\mu(\cap_{\ell \in I} U_\ell)=\prod_{\ell\in I} \mu(U_\ell)$ for any finite subset $I$ of $\calS$).   By Proposition~\ref{P:key}(\ref{a:lower}), we have
\[
 \sum_{\ell \in \calS} \mu(U_\ell) \geq c \sum_{\ell \in \calS} \frac{1}{\ell}=+\infty
\]
where the last equality uses that $\calS$ has positive density.   The second Borel-Cantelli lemma now implies that the set $\bigcap_{n=1}^\infty \bigcup_{\ell \geq n,\, \ell\in\calS} U_\ell$ has $\mu$-measure $1$.  Equivalently,
\[
\mu\big(\big\{ \sigma\in \beta\Gal_L : A(\Kbar(\sigma))[\ell]\neq 0 \text{ for infinitely many primes $\ell \in \calS$}\}\big) = 1.
\]
By combining the $[L:K]$ cosets of $\Gal_L$ in $\Gal_K$, we find that
\[
\mu\big(\big\{ \sigma\in \Gal_K : A(\Kbar(\sigma))[\ell]\neq 0 \text{ for infinitely many primes $\ell\in\calS$}\}\big) = [L:K].
\]
Theorem~\ref{T:main} follows by recalling that $\mu_K=[L:K]^{-1}\mu$.

\subsection*{Acknowledgements}
Special thanks to Moshe Jarden for introducing his and Geyer's conjecture to me and suggesting that I should try to study it.

\section{Counting points} \label{S:counting}
In this section, we give a quick application of the Weil conjectures.    The essential feature of the bound in the following theorem is its uniformity; its proof requires a bound for the sum of Betti numbers due to Katz (which builds on estimates of Bombieri).

\begin{thm} \label{T:counting}
Let $V\subseteq \AA^n_{\FF_q}$ with $n>1$ be a closed subvariety defined by the simultaneous vanishing of $r$ polynomials $f_1,\ldots, f_r$ in $\FF_q[x_1,\ldots,x_n]$, each of degree at most $d$.    Let $V_1,\ldots, V_m$ be the irreducible components of $V_{\FFbar_q}$ which have the same dimension as $V$.   Then 
\[
|V(\FF_q)| \leq m q^{\dim V} + 6 (3+rd)^{n+1}2^r q^{\dim V-1/2}.
\]
If the components $V_1,\ldots, V_m$ are all defined over $\FF_q$, then
\[
\Big|\, |V(\FF_q)| - m q^{\dim V} \Big| \leq 6 (3+rd)^{n+1}2^r q^{\dim V-1/2}.
\]
\end{thm}
\begin{proof}
Set $N=\dim V$ and fix a prime $\ell$ that does not divide $q$.  By the Grothendieck-Lefschetz theorem \cite{MR0463174}*{II Th\'eor\`eme~3.2}, we have
\[
|V(\FF_q)| = \sum_{i=0}^{2N} (-1)^i \Tr(F^*| H^i_c(V_{\FFbar_q},\QQ_\ell)).
\]
where $F^*$ is the linear transformation arising from the Frobenius morphism which acts on the $\ell$-adic cohomology groups with compact support.  By Deligne \cite{MR601520}, the eigenvalues of $F^*$ acting on $H^i_c(V_{\FFbar_q},\QQ_\ell)$ have absolute value at most $q^{i/2}$ (under any inclusion $\Qbar_\ell\subseteq \CC$).   Therefore, 
\begin{align}  \label{E:Weil argument}
\Big|\, |V(\FF_q)| - \Tr(F^*| H^{2N}_c(V_{\FFbar_q},\QQ_\ell))\, \Big|& \leq \sum_{i=0}^{2N-1} q^{i/2} \dim H^i_c(V_{\FFbar_q},\QQ_\ell)\\
\notag &\leq q^{N-1/2} \sum_{i=0}^{2N-1} \dim H^i_c(V_{\FFbar_q},\QQ_\ell)) \\
\notag & \leq  q^{N-1/2}\cdot6  (3+rd)^{n+1}2^r
\end{align}
where the last inequality follows by the Corollary of Theorem~1 of \cite{MR1803934}.

First suppose that the components $V_1,\ldots, V_m$ are all defined over $\FF_q$.  Choose a closed subvariety $Z$ with $\dim Z< \dim V=N$ such that $U:=V-Z$ is the disjoint union of smooth, open and absolutely irreducible subvarieties $U_1,\ldots, U_m$ of $V$.  We have an excision exact sequence 
\[
H^{2N-1}_c(Z_{\FFbar_q},\QQ_\ell) \to H^{2N}_c(U_{\FFbar_q},\QQ_\ell) \to H^{2N}_c(V_{\FFbar_q},\QQ_\ell) \to H^{2N}_c(Z_{\FFbar_q},\QQ_\ell).
\]
The strict inequality $\dim Z < N$ implies that $H^{i}_c(Z_{\FFbar_q},\QQ_\ell)=0$ for all $i > 2(N-1)$, so we have a natural isomorphism $H^{2N}_c(U_{\FFbar_q},\QQ_\ell) \xrightarrow{\sim} H^{2N}_c(V_{\FFbar_q},\QQ_\ell)$ with compatible linear maps $F^*$.  We have $H^{2N}_c(U_{\FFbar_q},\QQ_\ell)=\oplus_{i=1}^m H^{2N}_c(U_{i,\FFbar_q},\QQ_\ell)$.   Using that $U_i$ is smooth and absolutely irreducible, Poincar\'e duality gives an isomorphism $H^{2N}_c(U_{i,\FFbar_q},\QQ_\ell) =\QQ_\ell(-N)$ for all $i$.  Therefore, $H^{2N}_c(V_{\FFbar_q},\QQ_\ell)\cong \QQ_\ell(-N)^m$ and hence $F^*$ acts as multiplication by $q^N$ on this vector space.   By (\ref{E:Weil argument}), we now deduce that 
\begin{align*}
\Big|\, |V(\FF_q)| - m q^N \Big| & \leq  6  (3+rd)^{n+1}2^r q^{N-1/2}.
\end{align*}

Now suppose we are in the general case.   We have just shown that $\dim_{\QQ_\ell} H^{2N}_c(V_{\FFbar_q},\QQ_\ell) = m$ (one can first base extend by a finite extension of $\FF_q$ over which all of the $V_i$ are defined).   The eigenvalues of $F^*$ acting on $H^{2N}_c(V_{\FFbar_q},\QQ_\ell)$ have absolute value at most $q^{N}$ by Deligne, so $|\Tr(F^*| H^{2N}_c(V_{\FFbar_q},\QQ_\ell))|\leq mq^N$ and the theorem follows.
\end{proof}

\begin{remark}
For the main application in this paper, it would suffice to have a version of Theorem~\ref{T:counting} where the term $6 (3+rd)^{n+1}2^r$ is replaced by any constant depending only on $r$, $d$ and $n$.   Such bounds can be readily deduced from the Weil-Lang bounds instead of the more sophisticated cohomological machinery.   The above stronger version will be required in future work.
\end{remark}

\section{Proof of Proposition~\ref{P:key}}  \label{S:Key proof}

Fix an abelian variety $A$ of dimension $g\geq 1$ defined over a number field $K$.    For each rational prime $\ell$, let $\rho_{A,\ell}\colon \Gal_K \to \Aut(A(\Kbar)[\ell])\cong \GL_{2g}(\FF_\ell)$ be the Galois representation coming from the Galois action on the $\ell$-torsion points of $A$.     For each $\ell$, let $\rho_{A,\ell^\infty}\colon \Gal_K\to \Aut(A(\Kbar)[\ell^\infty])\cong \GL_{2g}(\ZZ_\ell)$ be the $\ell$-adic representation which describes the Galois action on $A(\Kbar)[\ell^\infty]$.  

For a finite extension $K'$ of $K$ and a maximal ideal $\p$ of $\OO_{K'}$ such that $A_{K'}$ has good reduction,  let $P_{A,\p}(x)\in \ZZ[x]$ be the characteristic polynomial of Frobenius for the reduction of $A$ modulo $\p$; it is the unique polynomial in $\ZZ[x]$ such that $P_{A,\p}(x)= \det(xI-\rho_{A,\ell^\infty}(\Frob_\p))$ for all primes $\ell$ satisfying $\p\nmid \ell$.

\subsection{\texorpdfstring{Image of Galois modulo $\ell$}{Image of Galois modulo l}} \label{SS:Galois image}

\begin{thm}[Serre]  \label{T:Serre}
There is a finite Galois extension $L$ of $K$ and positive integers $N$, $r$ and $\kappa$ such that the following hold:
\begin{alphenum}
\item \label{I:Serre a}
For all $\ell\geq \kappa$, there is a connected, reductive subgroup $H_\ell$ of $\GL_{2g,\FF_\ell}$ of rank $r$ such that $\rho_{A,\ell}(\Gal_L)$ is contained in $H_\ell(\FF_\ell)$ and the index $[H_\ell(\FF_\ell): \rho_{A,\ell}(\Gal_L)]$ divides $N$.  Furthermore, $H_\ell$ contains the group $\GG_m$ of homotheties.
\item \label{I:Serre b}
The homomorphism $\prod_{\ell} \rho_{A,\ell} \colon \Gal_L \to \prod_{\ell} \rho_{A,\ell}(\Gal_L)$ is surjective.
\end{alphenum}
\end{thm}
The above theorem is a consequence of results of J.-P.~Serre presented in his 1985-1986 course at the Coll\`ege de France, see \cite{MR965792}.   Detailed sketches were supplied in letters that have since been published in his collected papers; see the beginning of \cite{MR1730973}, in particular the letters to M.-F.~Vign\'eras \cite{MR1730973}*{\#137} and K.~Ribet \cite{MR1730973}*{\#138} contain information on parts (\ref{I:Serre a}) and (\ref{I:Serre b}), respectively.   The paper \cite{MR1944805} contains a detailed construction of the reductive groups $H_\ell$ (where they are denoted by $G(\ell)^{\operatorname{alg}}$).   For the rest of \S\ref{S:Key proof}, we will use the notation of Theorem~\ref{T:Serre}.

\begin{lemma} \label{L:split torus}   There is a finite Galois extension $M$ of $\QQ$ such that if $\ell$ is a sufficiently large prime that splits completely in $M$, then the following hold:
\begin{alphenum}
\item \label{I:split torus i}
The reductive group $H_\ell$ is split.
\item \label{I:split torus ii}  Let $x_{i,j}$ ($1\leq i,j \leq 2g$) and $y$ be independent variables.  We may identify $\GL_{2g,\FF_\ell}$ with the closed subvariety of $\Spec( \FF_\ell[ x_{i,j}, y]) = \AA^{n}_{\FF_\ell}$, with $n=4g^2+1$, defined by the equation $\det(x_{i,j})\cdot y=1$ (that is, identify a matrix $B$ with the $n$-tuple $((B_{i,j}), 1/\det(B))$).

Let $T$ be a split maximal torus of $H_\ell$.   Then the torus $T$, viewed as a closed subvariety of $\AA^n_{\FF_\ell}$, is defined by at most $C_1$ polynomials of degree at most $C_2$, where $C_1$ and $C_2$ are constants that do not depend on $\ell$.
\end{alphenum}
\end{lemma}
\begin{proof}
Define the scheme $\AA^{2g}_*=\AA^{2g-1}\times \GG_m$, and let $\cl \colon \GL_{2g} \to \AA^{2g}_*$ be the morphism that associates to a matrix $B$ the $2g$-tuple $(a_1,\ldots,a_{2g})$ where $\det(xI-B) = x^{2g}+a_1x^{2g-1} + \ldots + a_{2g-1} x + a_{2g}$.     If $G$ is a connected reductive subgroup of $\GL_{2g,K}$ for a field $K$, then $\cl(G)$ is a closed irreducible subvariety of $\AA_{*,K}^{2g}$ whose dimension is the same as the rank of $G$ (it suffices to consider only a maximal torus of $G$). 

There is a finite extension $L'$ of $L$ such that the Zariski closure of $\rho_{A,\ell^\infty}(\Gal_{L'})$ in $\GL_{2g,\QQ_\ell}$ is a \emph{connected} algebraic group for each $\ell$, cf.~\cite{MR1730973}*{p.18} and \cite{MR1441234}.
Let $\calP$ be the Zariski closure in $\AA^{2g}_{*,\QQ}$ of the set of tuples $P_{\p}:=(a_1,\ldots, a_{2g})$ where $P_{A,\p}(x)=x^{2g}+a_1x^{2g-1} + \ldots + a_{2g-1} x + a_{2g}$ and $\p$ varies over the maximal ideals of $\OO_{L'}$ for which $A$ has good reduction.  Serre has shown that, after choosing an integral model of $\calP$, we have $\cl(H_\ell)=\calP_{\FF_\ell}$ for all sufficiently large $\ell$, see \cite{MR1730973}*{\#137 \S6}.   In particular, the rank of $H_\ell$ agrees with $\dim (\calP)$ for $\ell$ large enough (this is how $r$ is determined in the proof of Thereom~\ref{T:Serre}).

Let $d$ be the maximum number of distinct roots $P_{A,\p}(x)$ has in $\Qbar$ as $\p$ varies over the maximal ideal of $\OO_{L'}$ for which $A$ has good reduction.   For $\ell$ large enough so that $H_\ell$ is defined, we define $d_\ell$ to the maximum number of distinct roots $\det(xI-h) \in \FFbar_\ell[x]$ has as $h$ varies over the elements of $H_\ell(\FFbar_\ell)$.   For $\ell$ large, the equality $\cl(H_\ell)=\calP_{\FF_\ell}$ implies that $d=d_\ell$ (the polynomials with less than $d$ roots are described by a codimension $1$ subvariety of $\calP$).      The set of maximal ideals $\p\subseteq \OO_{L'}$ for which $P_{A,\p}(x)$ has $d$ distinct roots has density 1.  Let $\frakq$ be a maximal ideal of $\OO_{L'}$ for which $A$ has good reduction and $P_{A,\frakq}(x)$ has $d$ distinct roots.    There is a constant $c_1$ such that $P_{A,\frakq}(x)\equiv \det(xI - \rho_{A,\ell}(\Frob_{\frakq})) \in \FF_\ell[x]$ has $d=d_\ell$ distinct roots for all $\ell\geq c_1$.    Let $M$ be the splitting field of $P_{A,\frakq}(x)$ over $\QQ$.    For the rest of the proof, suppose that $\ell$ is a prime greater than $c_1$ for which $\ell$ splits completely in $M$, and hence $\det(xI - \rho_{A,\ell}(\Frob_{\frakq}))\in \FF_\ell[x]$ has $d$ distinct roots in $\FF_\ell$.  

Let $t_{\frakq}\in H_\ell(\FF_\ell)$ be the semisimple part of a representative of the conjugacy class $\rho_{A,\ell}(\Frob_{\frakq})$.  Let $T$ be a maximal torus of $H_\ell$ which contains $t_{\frakq}$; we will show that $T$ is split.  Let $X(T)$ be the group of characters $T_{\FFbar_\ell} \to \GG_{m,\FFbar_\ell}$ and let $\iota\colon T \to \GL_{2g,\FF_\ell}$ the inclusion morphism.  For each character $\alpha \in X(T)$, define the vector space
\[
V(\alpha) = \{ v \in \FFbar_\ell^{2g} :  \iota(t)\cdot v =  \alpha(t) v \text{ for all $t\in T(\FFbar_\ell)$}\}.
\]
We say that $\alpha$ is a \emph{weight} of $\iota$ if $V(\alpha)\neq 0$, and we will denote the (finite) set of such weights by $\Omega$.  We have $\FFbar_\ell^{2g} = \oplus_{\alpha \in \Omega} V(\alpha)$ and for each $t\in T(\FFbar_\ell)$, 
\[
\det(xI -\iota(t))=\prod_{\alpha\in \Omega}(x-\alpha(t))^{\dim_{\FFbar_\ell} V(\alpha)}.
\] 
Since every semisimple element of $H_\ell$ is conjugate to an element in $T$, we find that $|\Omega|=d_\ell$ and hence $|\Omega|=d$.   Since $P_{A,\frakq}(x)\equiv \det(xI - t_{\frakq})\in \FF_\ell[x]$ has $d$ distinct roots in $\FF_\ell$, we deduce that $\alpha(t_{\frakq})$ belongs to $\FF_\ell$ for each $\alpha\in \Omega$ and that $\alpha_1(t_{\frakq})\neq \alpha_2(t_{\frakq})$ for all distinct $\alpha_1,\alpha_2\in \Omega$.  

For $\sigma\in \Gal_{\FF_\ell}$ and $\alpha\in X(T)$, we define ${}^\sigma\! \alpha$ to be the character of $T$ for which $\sigma(\alpha(t))= {}^\sigma\! \alpha(\sigma(t))$ for all $t\in T(\FFbar_\ell)$; this defines an action of the absolute Galois group $\Gal_{\FF_\ell}=\Gal(\FFbar_\ell/\FF_\ell)$ on the character group $X(T)$.    Since $\iota$ is defined over $\FF_\ell$, $\Gal_{\FF_\ell}$ also acts on the set $\Omega$.  Take any $\alpha\in \Omega$ and $\sigma \in \Gal_{\FF_\ell}$.  Since $\alpha(t_{\frakq})$ and $t_{\frakq}$ are defined over $\FF_\ell$, we have $\alpha(t_{\frakq})=\sigma( \alpha(t_{\frakq})) = {}^\sigma\! \alpha(\sigma(t_{\frakq}))= {}^\sigma\! \alpha(t_{\frakq})$.  Since $\beta(t_{\frakq})$ takes distinct values for different $\beta \in \Omega$, we deduce that ${}^\sigma\! \alpha =\alpha$.    Therefore, the action of $\Gal_{\FF_\ell}$ on $\Omega$ is trivial.  The group $X(T)$ is generated by $\Omega$, since $\iota$ is a faithful embedding, so the $\Gal_{\FF_\ell}$ action on $X(T)$ is also trivial.   That $\Gal_{\FF_\ell}$ acts trivially on $X(T)$ implies that $T$ is a split torus \cite{MR1102012}*{III \S8}.   This completes the proof of part (\ref{I:split torus i}).
 
We will now prove part (\ref{I:split torus ii}).   Since all split maximal tori of $H_\ell$ are conjugate by an element of $H_\ell(\FF_\ell)$, and conjugation does change the number or degree of the equations needed to define the torus, we need only verify (\ref{I:split torus ii}) for our specific split torus $T$.   Similarly by conjugating $H_\ell$ by an element of $\GL_{2g}(\FF_\ell)$, we may assume that the split torus $T$ lies in the diagonal of $\GL_{2g,\FF_\ell}$.   Moreover, we may assume that the inclusion $T\to \GL_{2g,\FF_\ell}$ maps $t\in T$ to the diagonal matrix
\[
\left(\begin{smallmatrix}\alpha_1(t) I_{m_1} &  &  &  \\ & \alpha_2(t) I_{m_2} &  &  \\ &  & \ddots &  \\ &  &  & \alpha_d(t) I_{m_d}\end{smallmatrix}\right)
\]
where $\Omega=\{\alpha_1,\ldots, \alpha_d\}$ and $m_i = \dim_{\FFbar_\ell} V(\alpha_i)$.    Define $e_s = 1+ \sum_{k<s} m_k$.  The torus $T$ thus consists of the matrices $B\in \GL_{2g,\FF_\ell}$ for which $B_{i,j}=0$ for $i\neq j$,  $B_{i,i}=B_{j,j}$ if $e_{s-1} \leq i<j < e_s$ for $1\leq s\leq d$, and $\prod_{1\leq i \leq d} B_{i,i}^{n_{i}}=1$ whenever $\prod_{1\leq i \leq d} \alpha_i^{n_{i}}=1$ with $n_i\in\ZZ$.   It thus suffices to prove that subgroup $\mathcal{N}$ of $\ZZ^d$ consisting of those $(n_1,\ldots,n_d)$ for which $\prod_{1\leq i \leq d} \alpha_i^{n_i}=1$ is generated by the finite set $\{ (n_1,\ldots, n_d) : |n_i| \leq C \}$ where $C$ is some constant that does not depend on $\ell$.   One of the ingredients in Serre's proof of $\cl(H_\ell)=\calP_{\FF_\ell}$ for large $\ell$ is that we can lift $H_\ell$ to a reductive group $\calH_\ell\subseteq \GL_{2g,\ZZ_\ell}$ over $\ZZ_\ell$.  Moreover, our lifts can be chosen such that for any embedding $\QQ_\ell \hookrightarrow \CC$, the reductive group $H_{\ell,\CC} \subseteq \GL_{2g,\CC}$ is conjugate in $\GL_{2g,\CC}$ to a \emph{finite number} of reductive groups (which do not depend on $\ell$).   This finiteness allows us to pick a constant $C$ that depends only on these finitely many reductive groups, and is hence independent of $\ell$.
\end{proof}

\subsection{Proof of Proposition~\ref{P:key}}
With notation as in \S\ref{SS:Galois image}, fix a conjugacy class $C$ of $\Gal(L/K)$.  We define $d_C$ to be the maximum number of distinct roots $P_{A,\p}(x^N)$ has in $\Qbar$ as $\p$ varies over all primes of $\OO_K$ for which $A$ has good reduction, is unramified in $L$, and satisfies $(\p,L/K)=C$; fix such a prime $\p_C$ for which this maximum occurs.\\

Let $\calS$ be the set of primes $\ell$ that satisfy the following conditions:
\begin{itemize}
\item $\ell \geq \kappa$ and $\p_C\nmid \ell$ for each conjugacy class $C$ of $\Gal(L/K)$,
\item $\ell$ splits completely in $M$, 
\item For each conjugacy class $C$ of $\Gal(L/K)$, $P_{A,\p_C}(x^N) \bmod{\ell} \in \FF_\ell[x]$ has $d_C$ distinct roots in $\FF_\ell$.
\end{itemize}
The set $\calS$, after possibly removing a finite number of primes, will be the set of Proposition~\ref{P:key}.  The set $\calS$ has positive density by the Chebotarev density theorem.   After removing a finite number of primes from $\calS$, by Lemma~\ref{L:split torus}(\ref{I:split torus i}) we may assume that $H_\ell$ is split for all $\ell \in\calS$.\\

For the rest of this section, fix a prime $\ell\in \calS$ and an element $\beta \in \Gal_K$.   Let $C$ be the conjugacy class of $\Gal(L/K)=\Gal_K/\Gal_L$ containing $\beta\Gal_L$.   Choose a matrix $B\in \rho_{A,\ell}(\beta\Gal_L)$ that lies in the conjugacy class $\rho_{A,\ell}(\Frob_{\p_C})$.  Since the index of $\rho_{A,\ell}(\Gal_L)$ in $H_\ell(\FF_\ell)$ divides $N$, we have $h^N \in \rho_{A,\ell}(\Gal_L)$ for all $h\in H_\ell(\FF_\ell)$.   In particular, $Bh^N\in \rho_{A,\ell}(\beta\Gal_L)$ for every $h\in H_\ell(\FF_\ell)$.  Therefore, 
\begin{align} \label{E:union over T}
\bigcup_{T \text{ split maximal torus of $H_\ell$}} \{ B t^N : t \in T(\FF_\ell) \text{ such that $\det(I- Bt^N)=0$ and $t^N$ is regular in $H_\ell$}\}
\end{align}
is a subset of $\{ h\in \rho_{A,\ell}(\beta\Gal_L): \det(I-h)=0\}$.  Suppose that $t_1$ and $t_2$ are semisimple elements of $H_\ell(\FF_\ell)$ with $t_1^N$ and $t_2^N$ regular in $H_\ell$.  If $Bt_1^N=Bt_2^N$, then $t_1^N=t_2^N$, and since they are regular they must lie in a unique maximal torus of $H_\ell$; in particular, $t_1$ and $t_2$ lie in the same (unique) maximal torus of $H_\ell$.   Therefore, (\ref{E:union over T}) is actually a disjoint union.   

If $h$ is an element of the rank $r$ torus $T$, then there are at most $N^r$ element $t$ in $T$ for which $t^N=h$.   We thus have
\begin{align} \label{E:sum over T}
&|\{ h\in \rho_{A,\ell}(\beta\Gal_L): \det(I-h)=0\}| \\
\notag \geq& \frac{1}{N^r} \sum_{T} |\{  t \in T(\FF_\ell) :  \det(I- Bt^N)=0 \text{ and $t^N$ is regular in $H_\ell$}\}|
\end{align}
where the sum is over all split maximal tori $T$ of $H_\ell$.  The key technical lemma of this paper is the following:

\begin{lemma} \label{L:technical}
There is a constant $c$ not depending on the choice of $B$ or $\ell$ such that 
\begin{equation*}
|\{  t \in T(\FF_\ell) :  \det(I- Bt^N)=0 \text{ and $t^N$ is regular in $H_\ell$}\}| \geq \ell^{r-1} - c \ell^{r-3/2}
\end{equation*}
for all split maximal tori $T$ of $H_\ell$.  
\end{lemma}

Assuming the validity of Lemma~\ref{L:technical}, let us finish the proof of Proposition~\ref{P:key}.   Combining (\ref{E:sum over T}) with Lemma~\ref{L:technical}, we find that 
\[
|\{h\in \rho_{A,\ell}(\beta\Gal_L): \det(I-h)=0\}| \geq  \frac{1}{N^r} \sum_{T \text{ split maximal torus of }H_\ell}  \big(\ell^{r-1} - c \ell^{r-3/2}\big).
\]

Fix a split maximal torus $T$ of $H_\ell$ (such a torus exists by our choice of $\calS$).   All split maximal tori of $H_\ell$ are conjugate to $T$ by some element of $H_\ell(\FF_\ell)$.    Let $\mathcal{N}$ be the group of elements of $H_\ell(\FF_\ell)$ that normalize the torus $T$.   The group $\mathcal{N}$ clearly contains $T(\FF_\ell)$ and the quotient $\mathcal{N}/T(\FF_\ell)$ is isomorphic to the Weyl group $W(H_\ell)$.   Therefore, there are exactly $|H_\ell(\FF_\ell)|/|\mathcal{N}| = |H_\ell(\FF_\ell)| |W(H_\ell)|^{-1} (\ell-1)^{-r}$ split maximal tori of $H_\ell$.   Combining this with our previous estimate, we have
\begin{align*}
|\{h\in \rho_{A,\ell}(\beta\Gal_L): \det(I-h)=0\}| & \geq   N^{-r} |H_\ell(\FF_\ell)| |W(H_\ell)|^{-1} (\ell-1)^{-r} \cdot (\ell^{r-1} - c \ell^{r-3/2}\big)\\
&\geq N^{-r} |H_\ell(\FF_\ell)| |W(H_\ell)|^{-1} (1 - c \ell^{-1/2}\big)\cdot \ell^{-1}.
\end{align*}
Using that $|H_\ell(\FF_\ell)|\geq |\rho_{A,\ell}(\Gal_L)|=|\rho_{A,\ell}(\beta\Gal_L)|$, we find that 
\[
\frac{|\{h\in \rho_{A,\ell}(\beta\Gal_L): \det(I-h)=0\}|}{|\rho_{A,\ell}(\beta\Gal_L)|} \geq N^{-r}  |W(H_\ell)|^{-1} (1 - c \ell^{-1/2}\big)\cdot \ell^{-1}.
\]
Since $H_\ell$ is a reductive group of rank $r$, there is a lower bound for $|W(H_\ell)|^{-1}$ that depends only on $r$.    Proposition~\ref{P:key}(\ref{a:lower}) is now immediate after removing a finite number of primes from $\calS$.  Proposition~\ref{P:key}(\ref{b:independence}) is a consequence of Theorem~\ref{T:Serre}(\ref{I:Serre b}) and our choice of $L$.

\subsection{Proof of Lemma~\ref{L:technical}}
Fix a split maximal torus $T$ of $H_\ell$.  Let $W$ be the closed subvariety of $T$ defined by the equation $\det(I-Bt^N)=0$ where $t\in T$.   

By Theorem~\ref{T:Serre}(\ref{I:Serre a}), $T$ contains the group $\GG_m$ of homotheties.  Let $\varphi\colon W \to T/\GG_m$ be the morphism obtained by composing the inclusion $W \hookrightarrow T$ with the quotient homomorphism.    Take any $t\in T(\FFbar_\ell)$, and let $\bbar{t}$ be the corresponding coset in $T/\GG_m$.  Then $\varphi^{-1}(\bbar{t})=\{ \lambda t : \lambda \in \FFbar_\ell,\, \det(I-\lambda^{N}Bt^N)=0 \}$, and hence $|\varphi^{-1}(\bbar{t})|$ equals the number of distinct roots of $\det(x^N - Bt^N)$ in $\FFbar_\ell$.   In particular, $\varphi$ is a finite morphism and we shall denote its degree by $d$.

\begin{lemma} \label{L:rational fiber}
Assuming $\ell\in\calS$ is sufficiently large, there exists an element $t\in T(\FF_\ell)$ such that $\varphi^{-1}(\bbar{t})$ consists of $d$ distinct points each belonging to $W(\FF_\ell)$.   
\end{lemma}
\begin{proof}
By our choice of $\p_C$, the polynomial $P_{A,\p_C}(x^N)$ has degree $d_C$.  Our set $\calS$ was chosen so that the polynomial
\[
P_{A,\p_C}(x^N)\equiv \det(x^NI-\rho_{A,\ell}(\Frob_{\p_C})) =\det(x^N I -B) \in \FF_\ell[x]
\]
has $d_C$ distinct roots all of which belong to $\FF_\ell$.    In terms of our morphism $\varphi$, this shows that $\varphi^{-1}(\bbar{I})$ consists of $d_C$ points each belonging to $W(\FF_\ell)$.   So $d_C\leq d$ and it remains to prove equality.

Let $V$ be the subvariety of $T$ consisting of those $t\in T$ for which $\det(x^NI- Bt^N)$ has strictly less than $d$ distinct roots.       Using Lemma~\ref{L:split torus}(\ref{I:split torus ii}) and Theorem~\ref{T:counting}, we find that $|V(\FF_\ell)|=O(\ell^{r-1})$ where the implicit constant does not depend on $B$ or $\ell$.   Since $|T(\FF_\ell)|=(\ell-1)^r$, we find that for all sufficiently large $\ell\in \calS$, the set $T(\FF_\ell)-V(\FF_\ell)$ is non-empty; so there is a $t_1\in T(\FF_\ell)$ such that $\det(x^NI- Bt_1^N)\in \FF_\ell[x]$ has exactly $d$ distinct roots in $\FFbar_\ell$.   Since the index $[H_\ell(\FF_\ell): \rho_{A,\ell}(\Gal_L)]$ divides $N$, we find that $t_1^N$ lies in $\rho_{A,\ell}(\Gal_K)$, and hence $Bt_1^N$ belongs to $\rho_{A,\ell}(\beta\Gal_K)$.   By the Chebotarev density theorem, there is a prime $\p\nmid \ell$ for which $A$ has good reduction, $(\p,L/K)=C$, and $Bt_1^N$ is in the conjugacy class $\rho_{A,\ell}(\Frob_\p)$.  Since $P_{A,\p}(x^N) \equiv \det(x^NI - Bt_1^N) \bmod{\ell}$ has $d$ distinct roots in $\FFbar_\ell$, the polynomial $P_{A,\p}(x^N)$ will have at least $d$ distinct roots in $\Qbar$.   From our definition of $d_C$, we deduce that $d\leq d_C$.    Therefore, $d=d_C$ as claimed.
\end{proof}

\begin{lemma} \label{L:component}
For $\ell \in \calS$ sufficiently large, each irreducible components of $W_{\FFbar_\ell}$ has dimension $r-1$ and is defined over $\FF_\ell$.
\end{lemma}
\begin{proof}
Let $W_1,\ldots, W_m$ be the irreducible components of $W_{\FFbar_\ell}$.  Each component $W_i$ has dimension $r-1$ by Krull's Hauptidealsatz.   So it remains to show that all of the $W_i$ are defined over $\FF_\ell$, at least for $\ell$ sufficiently large.   

Set $V:=(T/\GG_m)_{\FFbar_\ell}$.  For each $1\leq i \leq m$, let $\varphi_i$ be the morphism $\varphi|_{W_i}\colon W_i \to V$.    The morphism $\varphi_i$ is a cover (possibly ramified; one could make it \'etale by replacing $W_i$ and $V$ by Zariski open subsets).   Let $d_i$ be the degree of $\varphi_i$.  Let $Z$ be the Zariski closure of $\varphi(\bigcup_{i\neq j} W_i \cap W_j)$ in $V$.  Using that the $\varphi_i$ are covers, one can show that $Z\neq V$.   So for a generic $v\in V(\FFbar_\ell)$ outside $Z$, we have a disjoint union $\varphi^{-1}(v)=\bigcup_i \varphi_i^{-1}(v)$ with $d=|\varphi^{-1}(v)|$ and $d_i=|\varphi_i^{-1}(v)|$.  Therefore, $d=\sum_i d_i$.

Assuming $\ell\in\calS$ is sufficiently large, we can fix an element $t\in T(\FF_\ell)$ satisfying the conditions of Lemma~\ref{L:rational fiber}.   By our choice of $t$, the fiber $\varphi^{-1}(\bbar{t})=\bigcup_i \varphi_i^{-1}(\bbar{t})$ has $d$ distinct elements.   Since $|\varphi_i^{-1}(\bbar{t})|\leq d_i$ for each $1\leq i \leq m$ and $d=\sum_i d_i$, we deduce that $\varphi^{-1}(\bbar{t})$ is the disjoint union of the sets $\varphi_i(\bbar{t})$ and each $\varphi_i^{-1}(\bbar{t})$ consists of $d_i$ distinct elements.    The disjointness implies that each point in $\varphi^{-1}(\bbar{t})$ lies in a unique irreducible component of $W_{\FFbar_\ell}$.

Fix $1\leq i \leq m$.   Choose a point $w_i \in \varphi_i^{-1}(\bbar{t})$ (such a point exists since $\varphi_i^{-1}(\bbar{t})$ consists of $d_i\geq 1$ elements).   We have $w_i \in W(\FF_\ell)$ by our choice of $t$, so $w_i=\sigma(w_i) \in \sigma(W_i)$ for all $\sigma \in \Gal_{\FF_\ell}$.   Since $W_i$ is the unique irreducible component of $W_{\FFbar_\ell}$ that contains $w_i$, we deduce that $\sigma(W_i)=W_i$ for all $\sigma\in\Gal_{\FF_\ell}$ and hence $W_i$ is defined over $\FF_\ell$ as claimed.
\end{proof}

 By taking $\ell\in \calS$ sufficiently large, we may assume by Lemma~\ref{L:component} that all of the irreducible components of $W_{\FFbar_\ell}$ are defined over $\FF_\ell$ (by adjusting $c$ appropriately, it is easy to verify Lemma~\ref{L:technical} for the finitely many excluded primes).   {}From Lemma~\ref{L:split torus}(\ref{I:split torus ii}) and our choice of $\calS$, the split torus $T$ (viewed as a closed subvariety of $\AA^n_{\FF_\ell})$ is defined by a bounded number of equations of bounded degree (that is, bounded independent of the choice of $B$ and $\ell\in \calS$).   Theorem~\ref{T:counting} thus implies that 
\begin{equation} \label{E: lower achieved}
|\{t\in T(\FF_\ell) : \det(I-B t^N)=0\}|=|W(\FF_\ell)| \geq \ell^{r-1} + O(\ell^{r-3/2})
\end{equation}
where the implicit constant does not depend on the choice of $B$ or $\ell$.   

\begin{lemma} \label{L:count irregular}
For $\ell \in \calS$, we have $|\{t\in T(\FF_\ell): t \text{ is not regular in } H_\ell \}| = O(\ell^{r-1})$ where the implicit constant depends only on $r$.
\end{lemma}
\begin{proof}
Let $\Phi=\Phi(T,H_\ell)$ be the set of \emph{roots of $H_\ell$ relative to $T$}, see \cite{MR1102012}*{8.17}.   The roots $\Phi$ are a finite subset of the group $X(T)$ of characters $T_{\FFbar_\ell}\to \GG_{m,\FFbar_\ell}$; the characters of $T$ are actually defined over $\FF_\ell$ since $T$ is split.    By \cite{MR1102012}*{12.2}, an element $t\in T(\FF_\ell)$ is regular if and only if $\alpha(t)=1$ for all $\alpha \in \Phi$.   We thus have
\[
\{t\in T(\FF_\ell): t \text{ is not regular in } H_\ell \} = \bigcup_{\alpha\in \Phi} (\ker \alpha)(\FF_\ell).
\]
Each $\alpha\in \Phi$ is a non-trivial character defined over $\FF_\ell$, so the irreducible component of $\ker \alpha$ containing the identity is a split torus of rank $r-1$ defined over $\FF_\ell$.  Therefore, $|(\ker \alpha)(\FF_\ell)| \leq m_\alpha (\ell-1)^{r-1}$ where $m_\alpha$ is the number of irreducible components of $(\ker \alpha)_{\FFbar_\ell}$, and hence
\begin{equation} \label{E:irreg}
|\{t\in T(\FF_\ell): t \text{ is not regular in } H_\ell \}|\leq \sum_{\alpha \in \Phi} m_\alpha (\ell-1)^{r-1}\leq |\Phi| \, \big(\max_{\alpha \in \Phi} m_\alpha\big) (\ell-1)^{r-1}.   
\end{equation}

Fix a root $\alpha\in \Phi$.  Let $X(\ker \alpha)$ be the group of characters $(\ker\alpha)_{\FFbar_\ell}\to \GG_{m,\FFbar_\ell}$.  The cardinality of the torsion subgroup of $X(\ker\alpha)$ is divisible by $m_\alpha$.  The exact sequence, $1\to \ker\alpha \to T \xrightarrow{\alpha} \GG_m \to 1$, induces an exact sequence
\[
0\to \ZZ \xrightarrow{1\mapsto \alpha} X(T) \to X(\ker\alpha)\to 0
\]
and hence we have an isomorphism $X(T)/\ZZ\alpha \cong X(\ker \alpha)$.   The subgroup $\Psi$ of $X(T)$ generated by the roots $\Phi$ and a basis of characters of the center of $H_\ell$ is of bounded index, the bound depending only on the type of $H_\ell$ (see \cite{MR794307}*{1.11}).   So the order of the torsion subgroup of $X(\ker \alpha)$ varies with that of $\Psi/\ZZ\Phi$ by only a finite amount, which depends only on the type of $H_\ell$.   Since the group $\Psi/\ZZ\Phi$ and set $\Phi$ depend only on the root datum of $H_\ell$, we deduce that $|\Phi| \max_{\alpha \in \Phi} m_\alpha$ can be bounded in terms of the type of $H_\ell$ alone, and hence also in terms of the rank $r$.  The lemma follows by combining this with (\ref{E:irreg}). 
\end{proof}

Let $D$ be the set of $t\in T(\FF_\ell)$ for which $t^N$ is not regular in $H_\ell$.    For each $t' \in T(\FF_\ell)$, there are at most $N^r$ elements $t\in T(\FF_\ell)$ for which $t^N=t'$.     Thus by Lemma~\ref{L:count irregular}, we have
\begin{equation} \label{E:irregular}
|D| \leq N^r |\{t'\in T(\FF_\ell): t' \text{ is not regular in } H_\ell \}| = O(\ell^{r-1})
\end{equation}
where the implicit constant depends only on $r$ and $N$.   The group $\GG_m(\FF_\ell)=\FF_\ell^\times$ acts on $D$ since $\GG_m\subseteq T$.  For each $t\in D$, there are at most $d$ values of $\lambda\in \FF_\ell^\times$ such that $\lambda t\in W(\FF_\ell)$.   Therefore,
\begin{equation} \label{E:irregular W}
|\{t\in W(\FF_\ell): t^N \text{ not regular in $H_\ell$}\}|\leq d  |D|/(\ell-1) =O(\ell^{r-2})
\end{equation}
where the last equality follows from (\ref{E:irregular}) and the implicit constant depends only on $r$, $N$ and $d$.    The lemma now follows by combining (\ref{E: lower achieved}) and (\ref{E:irregular W}).

\bibliographystyle{plain}

\begin{bibdiv}
\begin{biblist}

\bib{MR1102012}{book}{
      author={Borel, A.},
       title={Linear algebraic groups},
     edition={Second},
      series={Graduate Texts in Mathematics},
   publisher={Springer-Verlag},
     address={New York},
        date={1991},
      volume={126},
}

\bib{MR794307}{book}{
      author={Carter, R.~W.},
       title={Finite groups of {L}ie type},
      series={Pure and Applied Mathematics (New York)},
   publisher={John Wiley \& Sons Inc.},
     address={New York},
        date={1985},
        note={Conjugacy classes and complex characters, A Wiley-Interscience
  Publication},
}

\bib{MR601520}{article}{
      author={Deligne, P.},
       title={La conjecture de {W}eil. {II}},
        date={1980},
     journal={Inst. Hautes \'Etudes Sci. Publ. Math.},
      number={52},
       pages={137\ndash 252},
}


\bib{MR0463174}{book}{
      author={Deligne, P.},
       title={Cohomologie \'etale},
      series={Lecture Notes in Mathematics, Vol. 569},
   publisher={Springer-Verlag},
     address={Berlin},
        date={1977},
        note={S{\'e}minaire de G{\'e}om{\'e}trie Alg{\'e}brique du Bois-Marie
  SGA 4 1/2, Avec la collaboration de J. F. Boutot, A. Grothendieck, L. Illusie
  et J. L. Verdier},
}

\bib{MR0337997}{article}{
      author={Frey, G.},
      author={Jarden, M.},
       title={Approximation theory and the rank of abelian varieties over large
  algebraic fields},
        date={1974},
     journal={Proc. London Math. Soc. (3)},
      volume={28},
       pages={112\ndash 128},
}

\bib{MR868860}{book}{
      author={Fried, M.~D.},
      author={Jarden, M.},
       title={Field arithmetic},
      series={Ergebnisse der Mathematik und ihrer Grenzgebiete (3)},
   publisher={Springer-Verlag},
     address={Berlin},
        date={1986},
      volume={11},
}

\bib{MR2111902}{article}{
      author={Geyer, W.-D.},
      author={Jarden, M.},
       title={Torsion of abelian varieties over large algebraic fields},
        date={2005},
     journal={Finite Fields Appl.},
      volume={11},
      number={1},
       pages={123\ndash 150},
}

\bib{MR516151}{article}{
      author={Geyer, W.-D.},
      author={Jarden, M.},
       title={Torsion points of elliptic curves over large algebraic extensions
  of finitely generated fields},
        date={1978},
     journal={Israel J. Math.},
      volume={31},
      number={3-4},
       pages={257\ndash 297},
}

\bib{MR1831454}{article}{
      author={Jacobson, M.},
      author={Jarden, M.},
       title={Finiteness theorems for torsion of abelian varieties over large
  algebraic fields},
        date={2001},
     journal={Acta Arith.},
      volume={98},
      number={1},
       pages={15\ndash 31},
}

\bib{MR1803934}{article}{
      author={Katz, N.~M.},
       title={Sums of {B}etti numbers in arbitrary characteristic},
        date={2001},
     journal={Finite Fields Appl.},
      volume={7},
      number={1},
       pages={29\ndash 44},
}

\bib{MR1441234}{article}{
      author={Larsen, M.},
      author={Pink, R.},
       title={A connectedness criterion for {$l$}-adic {G}alois
  representations},
        date={1997},
     journal={Israel J. Math.},
      volume={97},
       pages={1\ndash 10},
}

\bib{MR1730973}{book}{
      author={Serre, J.-P.},
       title={{\OE}uvres. {C}ollected papers. {IV}},
   publisher={Springer-Verlag},
     address={Berlin},
        date={2000},
        note={1985--1998},
}

\bib{MR965792}{article}{
      author={Serre, J.-P.},
       title={R\'esum\'e des cours de 1985-1986},
        date={1986},
     journal={Annnuaire du Coll\`ege France},
       pages={95\ndash 100},
        note={(={\OE}uvres. {C}ollected papers. {IV}, 33--37)},
}

\bib{MR1944805}{article}{
      author={Wintenberger, J.-P.},
       title={D\'emonstration d'une conjecture de {L}ang dans des cas
  particuliers},
        date={2002},
     journal={J. Reine Angew. Math.},
      volume={553},
       pages={1\ndash 16},
}

\bib{MR885780}{article}{
      author={Zarhin, Yu.~G.},
       title={Endomorphisms and torsion of abelian varieties},
        date={1987},
     journal={Duke Math. J.},
      volume={54},
      number={1},
       pages={131\ndash 145},
}

\end{biblist}
\end{bibdiv}


\end{document}